\documentclass[12pt]{article}
\usepackage{amsmath,amssymb,amsthm, blindtext, mathtools}

\newcommand{\C}{{C}}

\long\def\symbolfootnote[#1]#2{\begingroup%
\def\thefootnote{\fnsymbol{footnote}}\footnote[#1]{#2}\endgroup}

\newtheorem{thm}{Theorem}[section]
\newtheorem{prop}[thm]{Proposition}
\newtheorem{lem}[thm]{Lemma}
\newtheorem{cor}[thm]{Corollary}

\theoremstyle{definition}

\newtheorem{example}[thm]{Example}

\theoremstyle{remark}

\newtheorem{rem}[thm]{Remark}
\usepackage{comment}

\title{H\"older  estimates for the $\bar\partial$ problem for  $(p,q)$ forms   on product domains}

\author{Yifei Pan and Yuan Zhang\footnote{partially supported by NSF DMS-1501024}}
\date{}

\begin{document}

\maketitle

\begin{abstract}
The purpose of this paper is to study  H\"older estimates for the  $\bar\partial$ problem  for  $(p,q)$ forms  on products of general planar domains. As indicated by an example of Stein and Kerzman, solutions to the $\bar\partial$ problem on product domains in $\mathbb C^n (n\ge 2)$  does not  gain  regularity in H\"older spaces. Making use of an integral representation of  Nijenhuis and Woolf,  we show that  given a $\bar\partial$-closed  $(p,q)$ form with $C^{k,\alpha}$ components, $0\le p\le n, 1\le q\le n$, $k\in \mathbb Z^+\cup \{0\}, 0<\alpha\le 1$, there is a  $C^{k, \alpha'}$ solution  to the  $\bar\partial$ problem on product domains  for any $0<\alpha'<\alpha$ with the desired H\"older estimate.
\end{abstract}

\renewcommand{\thefootnote}{\fnsymbol{footnote}}
\footnotetext{\hspace*{-7mm}
\begin{tabular}{@{}r@{}p{16.5cm}@{}}
Keywords. $\bar{\partial}$-equation, product domains, H\"older spaces.
\end{tabular}}

\section{ Introduction and the main theorems}

The existence and regularity of the Cauchy-Riemann equations have been thoroughly studied in literature along the line of H\"ormander's $L^2$ theory. An alternative approach is to express solutions in integral representations.  Through a series of work including Grauert-Lieb \cite{GL},
Henkin \cite{henkin},  Kerzman \cite{Kerzman}, Henkin-Romanov \cite{HR} and Diederich-Fischer-Forn{\ae}ss \cite{DFF},  supnorm and H\"older estimates of solutions were established   for  smooth bounded domains which are strongly pseudoconvex or convex of finite type. Higher order regularity of solutions on sufficiently smooth bounded strongly pseudoconvex or strongly $\mathbb C$-linearly convex domains were studied by Siu \cite{Siu}, Lieb-Range \cite{LR},  and more recently Gong \cite{Gong} and Gong-Lanzani \cite{GL1} et al.

 Let $\Omega\subset \mathbb C^n, n\ge 2$ be a product of bounded planer domains. Namely, $\Omega = D_1\times\cdots\times D_n$, where each $D_j\subset \mathbb C, j=1, \ldots, n, $ is a  bounded domain in $\mathbb C$ such that  $\partial D_j$ consists of a finite number of  rectifiable Jordan curves which do not intersect one another. Then $\Omega$ is a bounded  pseudoconvex domain (but not convex in general) with at most Lipschitz boundary. A solution operator to $\bar\partial$   was first constructed in a seminal work \cite{NW} of Nijenhuis and Woolf in an {\it iterated} H\"older  space  over polydiscs.   The supnorm estimate for $C^1$ data up to the boundary was proved by Henkin \cite{henkin2}   on the bidisc. Recently, 
 Chen-McNeal \cite{ChM}  studied a type of $L^p$-Sobolev estimates for $(0,1)$ forms on general product domains in $\mathbb C^2$. They  further showed that Henkin's solution operator is not bounded in $L^p, 1\le p<2$.   For  product domains of arbitrary dimensions, Fassina-Pan \cite{FP}  constructed a solution operator for $(0,1)$ forms through one-dimensional method, from which they obtained $L^\infty$ estimates for smooth data.  See also Bertrams \cite{B}, Ehsani \cite{Ehsani},
 Chakrabarti-Shaw \cite{CS}, Dong-Li-Treuer \cite{DLT} 
 and the references therein for investigation of the canonical solutions on product domains.

  We should point out that unlike strictly pseudoconvex smooth domains, the $\bar\partial$ problem on product domains does not gain regularity. Indeed, motivated by an example of Stein and Kerzman \cite{Kerzman}, one can construct examples  to show that the $\bar\partial$ problem on product domains in general has  no gain of regularity in the (standard) H\"older spaces. The  examples  are  verified at the end of  Section 5.  Therefore, a natural question is, given a H\"older data on product domains, whether there exists a  solution to the $\bar\partial$ equation in the same H\"older class.  It is our goal to generalize the result of \cite{NW} and study the classical H\"older estimate of a $\bar\partial$ solution operator for $(p,q)$ forms  on general product domains. 

Let $C^{k, \alpha}(\Omega)$ be the (standard) H\"older space, $ k\in \mathbb Z^+\cup \{0\}$, $ 0<\alpha\le 1, $ and $(p, q)$ form   is said to be in $ C_{(p,q)}^{k, \alpha}(\Omega)$, $0\le p\le n$, $1\le q\le n$, if all its components are in $ C^{k, \alpha}(\Omega)$. (See Section 2 for the definition.) Given a function $f\in C^{k, \alpha}(\Omega)$, define for $z\in \Omega$, the solid and boundary Cauchy type integrals below, respectively. 
\begin{equation}\label{TS}
  \begin{split}
    T_j f(z):&=-\frac{1}{2\pi i}\int_{D_j}\frac{f(z_1, \ldots, z_{j-1}, \zeta_j, z_{j+1}, \ldots, z_n)}{\zeta_j-z_j}d\bar\zeta_j\wedge d\zeta_j;\\
    S_j f(z):&=\frac{1}{2\pi i}\int_{\partial D_j}\frac{f(z_1, \ldots, z_{j-1}, \zeta_j, z_{j+1}, \ldots, z_n)}{\zeta_j-z_j}d\zeta_j.
  \end{split}
\end{equation}

The  boundedness of these operators was established by Nijenhuis and Woolf in \cite{NW}  on polydiscs with respect to an iterated H\"older norm, which is stronger than the (standard) H\"older norm. See Section 3 for a revisit of the related work in \cite{NW}. Thus the resulting iterated H\"older spaces are  subspaces of the corresponding (standard) H\"older spaces.    Since their approach  relies also largely on rich symmetry of polydiscs, the  method no longer works either for the standard H\"older spaces or over general product domains. In this paper, we  prove the H\"older regularity for $T_j$ and $S_j$ in the (standard) H\"older spaces on general product domains.  Indeed, as demonstrated by examples in Section 4 in contrast to 
their one dimensional counterparts on planar domains, the following H\"older estimates for $T_j$ and $S_j$   turn out to be optimal.   
\medskip

\begin{thm}\label{Holder}
  a). $T_j$ is a bounded linear operator sending $C^{k, \alpha}(\Omega)$ into $C^{k, \alpha}(\Omega)$, $k\in \mathbb Z^+\cup \{0\}$,  $0<\alpha<1$. Namely, there exists some constant $C$ dependent only on $\Omega, k$ and $\alpha$, such that for any $f\in C^{k, \alpha}(\Omega)$,
      \begin{equation}\label{Tb}
    \begin{split}
       \|T_j f\|_{C^{k, \alpha}(\Omega)}\le C\|f\|_{C^{k, \alpha}(\Omega)}.
          \end{split}
  \end{equation}
b).  $S_j$ is a bounded linear operator sending $C^{k, \alpha}(\Omega)$ into $C^{k, \alpha'}(\Omega)$, $k\in \mathbb Z^+\cup \{0\}$, $0<\alpha'<\alpha\le 1$. Namely, there exists some $C$ dependent only on $\Omega, k, \alpha$ and $ \alpha'$, such that for any $f\in C^{k, \alpha}(\Omega)$,
      \begin{equation}\label{Sb}
    \|S_j f\|_{C^{k, \alpha'}(\Omega)}\le C\|f\|_{C^{k, \alpha}(\Omega)}.
     \end{equation}
\end{thm}

\medskip

As an application of the  boundedness of these operators in H\"older spaces, an estimate of a $\bar\partial$ solution in H\"older spaces is obtained with a loss of regularity that can be made arbitrarily small as follows.
 
\begin{thm}\label{main}
Let $D_j\subset\mathbb C$, $j= 1, \ldots, n,$ be bounded domains with  $C^{k+1,\alpha}$ boundary,   $n\ge 2, k\in \mathbb Z^+\cup \{0\}, 0<\alpha\le 1$, and let $\Omega: = D_1\times\cdots\times D_n$. Assume that $\mathbf f\in C_{(p, q)}^{k, \alpha}(\Omega)$ is a $\bar\partial$-closed $(p, q)$ form on $\Omega$,  $0\le p\le n, 1\le q\le n$. There exists a solution  $\mathbf u\in C_{(p, q-1)}^ {k, \alpha'}(\Omega)$ to  $\bar\partial \mathbf u =\mathbf f$ 
 such that for any $0<\alpha'<\alpha$,  $\|\mathbf u\|_{C^{k, \alpha'}(\Omega)}\le C\|\mathbf f\|_{C^{k, \alpha}(\Omega)}$, where $C$ depends only on $\Omega, k, \alpha$ and $\alpha'$. Here when $k=0$, all equations are understood in the sense of distributions.
\end{thm}
\medskip

It is desirable to know whether  there exists a solution operator that can achieve the same regularity as that of the data in H\"older spaces. However, we do not have answers at this point. We also mention that   another type of an iterated H\"older space was studied in  \cite{ChM2} where estimates of the solutions depend on higher order derivatives of the data. See Remark \ref{mm} d) for a brief comparison of these spaces and the  corresponding estimates.

For smooth data up to the boundary of the product domains,  the existence of smooth solutions for $(p,1)$ forms  has  already been obtained in \cite{CS} with  Sobolev estimates. As a direct consequence of Theorem \ref{main}, we obtain the following corollary   for  $(p, q)$ forms smooth up to the boundary in terms of H\"older estimates.

\begin{cor}\label{mains}
  Let $D_j\subset\mathbb C$, $j= 1, \ldots, n$,  be bounded domains with  $C^{\infty}$ boundary, $ n\ge 2$, and $\Omega: = D_1\times\cdots\times D_n$. Assume $\mathbf f\in C_{(p, q)}^{\infty}(\overline\Omega)$ is a $\bar\partial$-closed $(p,q)$ form on $\Omega$, $0\le p\le n, 1\le q\le n$. There exists a solution $\mathbf u\in C_{(p, q-1)}^{\infty}(\overline\Omega)$ to $\bar\partial \mathbf u =\mathbf f$ in $\Omega$. Moreover,   for all $k\in \mathbb Z^+\cup \{0\}$, $0<\alpha'<\alpha \le 1$, $\|\mathbf u\|_{C^{k, \alpha'}(\Omega)}\le C_{k, \alpha, \alpha'}\|\mathbf f\|_{C^{k, \alpha}(\Omega)}$, where  $C_{k, \alpha, \alpha'}$ depends only on $\Omega, k, \alpha$ and $\alpha'$.
\end{cor}

The rest of the paper is organized as follows.  Section 2 addresses preliminaries about solid and boundary Cauchy  integrals on the complex plane. Section 3 is a revisit of the fundamental work of Nijenhuis and Woolf \cite{NW} on the $\bar\partial$ problem.  Theorem \ref{Holder} is proved in Section 4, along with examples demonstrating those estimates are optimal in H\"older category. The last section is  devoted to the proof of   Theorem \ref{main} and Corollary \ref{mains}. 
  In the Appendix, a  convergence result of the mollifier method in H\"older spaces is proved.
\medskip

\textbf{Acknowledgement: } Both authors  thank Liding Yao for providing an example in the Appendix.  Part of the work was done while the second author was visiting American Institute of Mathematics (AIM). She also appreciates  AIM and Association for Women in Mathematics (AWM) for hospitality during her visit.

\section{Notations and Preliminaries}
 As a common notice,  
 we use $u$ and $f$ to represent complex-valued functions, and boldface $\mathbf u$ and $\mathbf f$ to represent forms. Unless otherwise specified,  $C$ represents a constant dependent only on $\Omega, k,  \alpha$ and $\alpha'$, which may be of different values in different places.

Let $\Omega\subset\mathbb C^n$ be a bounded domain, the standard H\"older space    $C^{k, \alpha} (\Omega), k\in\mathbb Z^+\cup\{0\}, 0<\alpha\le 1$ is defined by
 \begin{equation*}
        \{f\in C^k(\Omega): \|f\|_{C^{k, \alpha}(\Omega)} := \|f\|_{C^k(\Omega)}+ \sum_{|\gamma|=k}H^\alpha[D^\gamma f] <\infty\}.
    \end{equation*}
    Here $D^\gamma$ represents any $|\gamma|$-th derivative operator, $$\|f\|_{C^k(\Omega)}: = \sum_{|\gamma|=0}^k\sup_{z\in \Omega}|D^\gamma f(z)| $$ and the H\"older semi-norm is
$$H^\alpha[f]: =\sup_{z, z'\in \Omega, z\ne z'} \frac{|f(z) - f(z')|}{|z-z'|^\alpha}.$$
When $k=0, 0<\alpha<1$, we write $ C^{0, \alpha} (\Omega) =  C^{\alpha} (\Omega)$. For a $(p,q)$ form $\mathbf f\in C^{k, \alpha}_{(p,q)}(\Omega)$, define $\|\mathbf f\|_{C^{k, \alpha}(\Omega)}$ to be the sum of the ${C^{k, \alpha}(\Omega)}$ norms of all its components.

 When $\Omega = D_1\times \cdots\times D_n$ is a product of planar domains,  for each $j\in \{1, \ldots, n\}$,  the H\"older semi-norm with respect to $j$-th variable for each fixed $(z_1,\ldots, z_{j-1},z_{j+1}, \ldots, z_n)\in D_1\times\cdots\times D_{j-1}\times  D_{j+1}\times\cdots\times D_n$  is defined by
 \begin{equation*}\begin{split}
   &H^\alpha_j[f(z_1,\ldots, z_{j-1},\cdot,z_{j+1}, \ldots, z_n)]:\\
    =&\sup_{\zeta, \zeta' \in D_j, \zeta\ne \zeta'} \frac{|f(z_1,\ldots, z_{j-1},\zeta,z_{j+1}, \ldots, z_n) - f(z_1,\ldots, z_{j-1},\zeta',z_{j+1}, \ldots, z_n)|}{|\zeta-\zeta'|^\alpha}.\end{split}
 \end{equation*}
 Clearly, $$\sup_{\substack{z_k\in D_k,\\ 1\le k(\ne j)\le n}}H^\alpha_j[f(z_1,\ldots, z_{j-1},\cdot,z_{j+1}, \ldots, z_n)]\le H^\alpha[f],\ \ \  j = 1. \ldots, n.$$ On the other hand,  the following  elementary lemma for H\"older functions is observed for product domains.

\medskip

\begin{lem}\label{comp} Let $\Omega = D_1\times \cdots\times D_n$ be a product of planar domains. Then
 \begin{equation}\label{ele}
     H^\alpha[f]\le \sum_{1\le j\le n}\sup_{\substack{z_k\in D_k,\\ 1\le k(\ne j)\le n}}H^\alpha_j[  f(z_1,\ldots, z_{j-1},\cdot,z_{j+1}, \ldots, z_n)].
 \end{equation}
   \end{lem}

\begin{proof}
    For simplicity of  exposition, assume  $n=2$ with $\Omega=D_1\times D_2$. Let $C$ be the right hand side of (\ref{ele}). For any $z=(z_1, z_2)\in D_1\times D_2, z'=(z'_1, z'_2)\in D_1\times D_2$, then $(z'_1, z_2)\in D_1\times D_2$. Hence
  $|f(z_1, z_2)- f(z'_1, z'_2)|\le |f(z_1, z_2)- f(z'_1, z_2)|+| f(z'_1, z_2)- f(z'_1, z'_2)|\le C |z -z'|^\alpha$.
  
\end{proof}

The rest of the section is devoted to  classical theory in complex analysis. Let $D$ be a bounded domain in  $\mathbb C$ with $C^{k+1,\alpha}$  boundary, $k\in \mathbb Z^+\cup \{0\}, 0<\alpha\le 1$. Given a complex-valued continuous function $f\in C(\bar D)$, we define the following two operators related to the Cauchy kernel for $z\in D$:

\begin{equation*}
\begin{split}
 Tf(z): &=\frac{-1}{2\pi i}\int_D \frac{f(\zeta)}{\zeta- z}d\bar{\zeta}\wedge d\zeta;\\
Sf(z): &=\frac{1}{2\pi i}\int_{\partial D}\frac{f(\zeta)}{\zeta- z}d\zeta.
\end{split}
\end{equation*}
Here  the positive orientation of $\partial D$ is adopted for the contour integral  such  that  $D$ is always to the left while traversing along the contour(s). As is well known, $T$ is  the universal solution operator for the $\bar\partial$ operator on $D$, while   $S$ turns integrable functions on $\partial D$ to holomorphic functions in $D$.
In the following, we state  some  properties of the two operators that will be used in later sections.

\begin{thm}(cf. \cite{V})\label{CG}
  Let D be a bounded domain with $C^{1, \alpha}$ boundary, $f\in C(\bar D)$ and $f_{\bar{z}} =\frac{\partial f}{\partial \bar z}\in L^p(D), p>2$. Then
  \begin{equation*}
    f =S f+ T(f_{\bar\zeta}) \ \ \text{in}\ \ D.
  \end{equation*}
\end{thm}


\begin{thm}(cf. \cite{V})
\label{holderC}
  Let D be a bounded domain with $C^{k+1, \alpha}$ boundary, and $f\in C^{k, \alpha}(D), k\in \mathbb Z^+\cup \{0\}, 0<\alpha<1$. Then $Tf\in C^{k+1, \alpha}(D)$ and $Sf\in C^{k, \alpha}(D)$. Moreover,  there exists a constant $C$ dependent only on $D, k$ and $\alpha$, such that
  \begin{equation*}
    \begin{split}
      &\|Tf\|_{C^{k+1, \alpha}(D)}\le C\|f\|_{C^{k, \alpha}(D)};\\
      &\|Sf\|_{C^{k, \alpha}(D)}\le C\|f\|_{C^{k, \alpha}(D)}.
    \end{split}
  \end{equation*}
\end{thm}


\begin{thm}(cf. \cite{V})\label{lpC}
   Let D be a bounded domain. Then $Tf\in C^{\alpha}(D)$ if $f\in L^p(D), p>2, \alpha=\frac{p-2}{p}$, and there exists a constant $C$ dependent only on $D$ and $p$, such that $$\|Tf\|_{C^{\alpha}(D)}\le C\|f\|_{L^p}.$$
  Moreover, $\bar\partial T = id$ on $L^p(D), 1\le p<\infty$ in the sense of distributions.
\end{thm}

For proofs of the above theorems, see    p. 41 \cite{V} for Theorem \ref{CG}; p. 56 \cite{V} and    p. 21 \cite{V} for Theorem \ref{holderC};
 p. 38 \cite{V} and p. 29 \cite{V}  for Theorem \ref{lpC}.

\section{Revisit of Nijenhuis-Woolf's work on polydiscs}

In this section, we  present the related results in the fundamental work of Nijenhuis and Woolf \cite{NW} for the $\bar\partial$ problem on the polydisc $\mathbb D^n: =\{(z_1,\cdots,  z_n)\in \mathbb C^n: |z_j|<1, j=1, \ldots, n\}$. We shall purposely retain their notation as much as possible for the convenience of readers.

Let $f$ be a complex-valued function on $\mathbb D^n$. Define $\triangle_i f$ to be a function on the subset $\mathbb D_i$ of  $\mathbb D^{n+1}$ whose points $Z_{i} = (z_1, \cdots, z_{i-1}, (z_i, z_i'), z_{i+1}, \cdots, z_n)$ satisfies $z_i\ne z_i'$,   such that 
$$\triangle_i f(Z_{i}) = f(z_1, \cdots,  z_i,  \cdots, z_n)- f(z_1, \cdots, z_{i-1},  z_i', z_{i+1}, \cdots, z_n).$$
Recursively, let $\mathbb D_{i_1\cdots i_k}$ be the subset of $\mathbb D^{n+k}$ whose points $Z_{i_1\cdots i_k} =(z_1, \cdots,(z_{i_1}, z_{i_1}'), \cdots, (z_{i_k}, z_{i_k}'),\\ \cdots, z_n)$ satisfy $ z_{i_j}\ne z_{i_j}', j=1, \cdots, k$. Define on $\mathbb D_{i_1\cdots i_k}$ a function   $$\triangle_{i_1\cdots i_k}f: = \triangle_{i_k}\triangle_{i_1\cdots i_{k-1}} f.$$

In \cite{NW},  a naturally defined iterated H\"older space $\mathcal C^\alpha(\mathbb D^n)$ (with the notation slightly different from that of the standard H\"older space) was introduced such that a function $f\in \mathcal C^\alpha (\mathbb D^n)$ if 
\begin{equation}\label{hh1}
    \|f\|_{ \mathcal C^\alpha (\mathbb D^n)}: = \|f\|_{C(\mathbb D^n)} +\sum_{k=1}^n H_\alpha^{(k)}[f]<\infty.
\end{equation}
Here
$$H_\alpha^{(k)}[f]: =\sup_{ \substack
{1\le i_1<\ldots <i_k\le n, \\Z_{i_1\cdots i_k}\in \mathbb D_{i_1\cdots i_k}}} \left\{\frac{|\triangle_{i_1\cdots i_k}f(Z_{i_1\cdots i_k}) |}{|z_{i_1}-z_{i_1}'|^\alpha\cdots|z_{i_k}-z_{i_k}'|^\alpha}\right\}. $$
Since $H_\alpha^{(1)}$ is precisely $H^\alpha$ in Section 2, we have $ \|\cdot\|_{C^\alpha (\mathbb D^n)}\le \|\cdot\|_{ \mathcal C^\alpha (\mathbb D^n)} $.  In fact, one further has (p. 485 \cite{NW})
\begin{equation}\label{incl}
    \mathcal C^\alpha(\mathbb D^n)\subset C^\alpha(\mathbb D^n)\subset \mathcal C^{\frac{\alpha}{n}}(\mathbb D^n). 
\end{equation}

\medskip

Let $T_j$ and $S_j$ be the solid and boundary Cauchy integral operators acting on functions over $j$-th slice of $\mathbb D^n$  as in (\ref{TS}).
 Given a $(p, q)$ form \begin{equation}\label{pq}
     \mathbf f = \sum_{\substack{i_1<\cdots<i_p,\\ j_1<\cdots<j_q}}f_{i_1\cdots i_p\bar j_1\cdots\bar j_q}dz_{i_1}\wedge \cdots \wedge dz_{i_p}\wedge d\bar z_{j_1}\wedge\cdots \wedge d\bar z_{j_q}\in C^1_{(p,q)}(\bar{\mathbb D}^n),
 \end{equation}
define  $T_j \mathbf f$ and $S_j\mathbf f$ to be  the action on the corresponding  component functions. Namely,
\begin{equation*}
    \begin{split}
       & T_j\mathbf f: = \sum_{ \substack{1\le i_1<\cdots<i_p\le n,\\1\le j_1<\cdots<j_q\le n}}T_jf_{i_1\cdots i_p\bar j_1\cdots\bar j_q}dz_{i_1}\wedge \cdots \wedge dz_{i_p}\wedge d\bar z_{j_1}\wedge\cdots \wedge d\bar z_{j_q};\\
       & S_j\mathbf f: = \sum_{\substack{1\le i_1<\cdots<i_p\le n,\\1\le j_1<\cdots<j_q\le n}}S_jf_{i_1\cdots i_p\bar j_1\cdots\bar j_q}dz_{i_1}\wedge \cdots \wedge dz_{i_p}\wedge d\bar z_{j_1}\wedge\cdots \wedge d\bar z_{j_q}.
    \end{split}
\end{equation*}

To construct a solution operator to the $\bar\partial$ equation for $(p, q)$ forms, \cite{NW} introduced  a projection operator $\pi_k$. Precisely speaking, for the $(p,q)$ form $\mathbf f$  given in (\ref{pq}) and each $1\le k\le n$,  $\pi_k\mathbf f$  is a $(p, q-1)$ form with
\begin{equation*}
    \pi_k\mathbf f: = (-1)^p\sum_{\substack{1\le i_1<\cdots<i_p\le n,\\ 1\le k<j_2<\cdots<j_q\le n}}f_{i_1\cdots i_p\bar k\bar j_2\cdots\bar j_q}dz_{i_1}\wedge \cdots \wedge dz_{i_p}\wedge d\bar z_{j_2}\wedge\cdots \wedge d\bar z_{j_q}. 
\end{equation*}
Based on these  definitions, a solution operator of the $\bar\partial$ equation for $(p,q)$ forms on polydisc was  constructed  in \cite{NW} (p. 430). 

\begin{thm}\cite{NW}\label{nw1}
If $\mathbf f\in  C^{1}_{(p,q)}(\bar{\mathbb D}^n)$ is $\bar\partial$-closed on $\mathbb D^n$, then 
\begin{equation}\label{key}
  T\mathbf f: = 
  T_1\pi_1 \mathbf f +T_2S_1\pi_2 \mathbf f+\cdots+ T_nS_1\cdots S_{n-1}\pi_n\mathbf f
\end{equation}
is a solution to $\bar\partial \mathbf u = \mathbf f$ on $\mathbb D^n$. 
\end{thm}

In terms of the norm estimates of the operators, \cite{NW} (p. 435 \& p. 487) proved the following fundamental boundedness for both $T_j$ and $S_j$ operators in the iterated H\"older spaces.

\begin{thm}\cite{NW}\label{nw2}
If $\mathbf f\in \mathcal C^{\alpha}_{(p,q)}(\mathbb D^n)$, then there exists a constant $C$ dependent only on $n$ and $\alpha$ such that 
\begin{equation*}
\begin{split}
    &\|T_j\mathbf f\|_{\mathcal C^{\alpha} (\mathbb D^n)}\le C \|\mathbf f\|_{\mathcal C^\alpha (\mathbb D^n)};\\
    &\|S_j\mathbf f\|_{\mathcal C^{\alpha} (\mathbb D^n)}\le C \|\mathbf f\|_{\mathcal C^\alpha (\mathbb D^n)}.\\
    \end{split}
\end{equation*}
Consequently, the solution operator $T$ defined in (\ref{key}) satisfies
\begin{equation*}
    \|T\mathbf f\|_{\mathcal C^{ \alpha} (\mathbb D^n)}\le C \|\mathbf f\|_{\mathcal C^\alpha (\mathbb D^n)}.
\end{equation*}
\end{thm}
\medskip

\begin{rem}\label{mm}
a). Theorem \ref{nw1} was initially constructed for polydiscs in \cite{NW}. In fact, in exactly the same way  there (p. 430 \cite{NW}), one can show that  (\ref{key}) solves the $\bar\partial$ problem pointwisely on  arbitrary product domains when the datum is $ C_{(p,q)}^{1}$ up to the boundary.

\medskip
 b).  We suspect that the approach used in the proof of Theorem \ref{nw2}  could  be applied to  general product domains, since the domain under consideration  in \cite{NW} was  exclusively polydiscs which carry rich symmetry. 
 \medskip
 
 c). As will be seen in Example \ref{2222}, the estimate of $S_j$ in Theorem \ref{nw2} fails if we replace $\mathcal C^\alpha (\mathbb D^n)$ by the (standard) H\"older space $ C^\alpha (\mathbb D^n)$. 
 \medskip
 
 d).  Another type of an iterated H\"older space $\Lambda_2^\alpha$  was defined by Chen and McNeal \cite{ChM2} for a product of two general bounded domains. In the context of a product of two planar  domains $D_1$ and $D_2$, 
 $$\Lambda_2^\alpha(D_1\times D_2): = \{f\in C(D_1\times D_2): \|f\|_{\Lambda_2^\alpha(D_1\times D_2)}=|f|_{C(D_1\times D_2)} +H_\alpha^{(2)}[f]<\infty  \}.$$ 
$\Lambda_2^\alpha$ is different from $\mathcal C^\alpha$ in that the sum part in (\ref{hh1}) for $\mathcal C^\alpha$  is replaced by the single term $H_\alpha^{(2)}[f]$ for $\Lambda_2^\alpha$. As a matter of fact, $ \|\cdot\|_{\mathcal C^\alpha(\mathbb D^2)}\ge \|\cdot\|_{\Lambda_2^\alpha(\mathbb D^2)}$ and so $\mathcal C^\alpha(\mathbb D^2)\subset \Lambda_2^\alpha(\mathbb D^2)$. 
 
 In \cite{ChM2}, it was  shown that  for any $\bar\partial$-closed $(0,1)$ form $\mathbf f = f_1d\bar z_1 +f_2d\bar z_2$ with  $f_1, f_2, \frac{\partial f_1}{\partial\bar z_2}\in \Lambda_2^\alpha(D_1\times D_2)$, there exists a solution $T\mathbf f\in \Lambda_2^\alpha(D_1\times D_2)$ to $\bar\partial u =\mathbf f$ such that
 $$\|T\mathbf f\|_{\Lambda_2^\alpha(D_1\times D_2)}\le C \left(\|\mathbf f\|_{\Lambda_2^\alpha(D_1\times D_2)} +\left\|\frac{\partial  f_1}{\partial\bar z_2}\right\|_{\Lambda_2^\alpha(D_1\times D_2)}\right).$$
 
  \cite{ChM2} compared $\Lambda_2^\alpha$ with (the standard) $C^\alpha$  by  constructing an example  in $\Lambda_2^\alpha({\mathbb D}^2)$ but not in $C^\alpha({\mathbb D}^2)$. They also attempted to find a Lipschitz function in $C^{0,1}({\mathbb D}^2)$ but not in $\Lambda_2^\alpha({\mathbb D}^2)$ for any $0<\alpha<1$. However,  this would contradict with (\ref{incl}) because for any $0<\alpha<1$, one necessarily has 
  $$ C^{0,1}({\mathbb D}^2)\subset \mathcal C^{\frac{\alpha}{2}}({\mathbb D}^2)\subset  \Lambda_2^{\frac{\alpha}{2}}({\mathbb D}^2).$$
  The mistake is due to the fact that  the constant $C$ in part (3) of Example 5.6 \cite{ChM2}  actually goes to $0$. Thus the limit there would not necessarily go to $\infty$ as they have claimed.
\end{rem}

\medskip

\section{Sharp H\"older bounds of the Cauchy type operators on  product domains}

Let $D_j\subset\mathbb C$, $j= 1, \ldots, n$, be a bounded domain with  $C^{k+1,\alpha}$ boundary, $n\ge 2$,   $k\in \mathbb Z^+\cup \{0\}, 0<\alpha\le 1$, and $\Omega: = D_1\times\cdots\times D_n$. 
Theorem \ref{holderC}-\ref{lpC} immediately imply the following lemma.
\medskip

\begin{lem}\label{123}
 There exists a constant $C$  dependent only on $\Omega, k$ and $\alpha$, such that for any $j=1,\ldots, n$, $f \in C^{k, \alpha}(\Omega)$,  $0<\alpha<1$, $k\in \mathbb Z^+\cup \{0\}$, $\gamma\in \mathbb Z^+\cup \{0\}$ with $\gamma\le k$, 
  \begin{equation*}
    \begin{split}
    \sup_{\substack{z_l\in D_l,\\ 1\le l(\ne j)\le n}} \|D_j^\gamma T_j f(z_1,\ldots, z_{j-1},\cdot,z_{j+1}, \ldots, z_n)\|_{C^{\alpha}(D_j)}\le&\left\{
      \begin{array}{cc}
      C\|f\|_{C(\Omega)}, &\gamma=0 \\
     C\|f\|_{C^{\gamma-1, \alpha}(\Omega)}, &\gamma\ge 1
     \end{array}
\right.  \le C\|f\|_{C^{\gamma, \alpha}(\Omega)};\\
    \sup_{\substack{z_l\in D_l,\\ 1\le l(\ne j)\le n}}\|D_j^\gamma S_j f(z_1,\ldots, z_{j-1},\cdot,z_{j+1}, \ldots, z_n)\|_{C^{\alpha}(D_j)}\le& C\|f\|_{C^{\gamma, \alpha}(\Omega)}.
    \end{split}
  \end{equation*}
  Here $D_j^\gamma$ represents any $\gamma$-th derivative operator with respect to the $j$-th variable.
\end{lem}

\medskip

\medskip

Although the solid Cauchy integral operator $T$ defined in Section 2 is a smoothing operator  in  dimension one, $T_j$ in (\ref{TS}) does not improve regularity along slice of higher dimensional domains, as demonstrated by the following example.

\begin{example}\label{1111}
Consider $f(z_1, z_2)= |z_2|^\alpha$ on $\mathbb D^2$. Then $f\in C^{\alpha}(\mathbb D^2)$. However a straight forward computation shows that $T_1f(z_1, z_2) = \bar z_1|z_2|^\alpha\notin C^{\alpha+\epsilon}(\mathbb D^2)$ for any $\epsilon>0$. 
\end{example}

On the other hand, in  contrast to the boundary Cauchy integral operator $S$ in one dimensional case, its counterpart $S_j$ in (\ref{TS}) no longer maintains H\"older regularity in higher dimensions. Indeed,  Tumanov (\cite{Tumanov} p.486) constructed the following concrete function $\tilde f\in C^\alpha(\partial\mathbb D\times \mathbb D)$ but $S_1\tilde f\notin C^\alpha(\partial\mathbb D\times \mathbb D), 0<\alpha<1$.

\begin{example}\label{2222}\cite{Tumanov} Define for $z_2\in \mathbb D$,
$$
  \tilde f(e^{i\theta}, z_2) =
  \left\{
      \begin{array}{cc}
     |z_2|^\alpha, & -\pi\le \theta \le -|z_2|^\frac{1}{2}; \\
     \theta^{2\alpha}, & -|z_2|^\frac{1}{2}\le \theta\le 0; \\
      \theta^\alpha, & 0\le \theta\le |z_2|;\\
      |z_2|^\alpha, &|z_2|\le \theta\le \pi.
    \end{array}
\right.$$
Then $\tilde f\in C^\alpha(\partial \mathbb D \times \mathbb D)$. However $S_1 \tilde f \notin C^\alpha(\partial\mathbb D\times \mathbb D)$.  Extend $\tilde f$ onto $\mathbb D^2$, denoted as $f$, such that $f\in C^\alpha(\mathbb D^2)$. One can check that  $S_1f\notin C^\alpha(\mathbb D^2)$. See \cite{PZ} for more details of the verification.
\end{example}

\medskip

In view of Example \ref{1111}-\ref{2222}, Theorem \ref{Holder}  characterizes  the optimal H\"older bounds of the two Cauchy type operators $T_j$ and $S_j$.
\medskip

\begin{proof}[Proof of Theorem \ref{Holder}] a).
We only prove (\ref{Tb}) when $j=1$ and $n=2$ for simplicity. The other cases are proved accordingly.

We first show $\|T_1f\|_{C^k(\Omega)}\le C\|f\|_{C^{k, \alpha}(\Omega)}$ for some constant $C$ independent of $f$. 
Write $D^\gamma = D^{\gamma_1}_1D^{\gamma_2}_2$,  $\gamma_1+\gamma_2\le k$.
 Then $D^\gamma T_1 f = D^{\gamma_1}_1 T_1(D_2^{\gamma_2} f)$. Hence by Lemma \ref{123}, $$\|D^\gamma T_1 f\|_{C(\Omega)} =  \sup_{z_2\in D_2}\|D^{\gamma_1}_1 T_1(D^{\gamma_2}_2 f)(\cdot, z_2)\|_{C(D_1)} \le  C\|D^{\gamma_2}_2 f\|_{C^{\gamma_1}(\Omega)}\le C\|f\|_{C^{k, \alpha}(\Omega)}.$$

Next, we show  $H^\alpha[D^\gamma T_1f]\le C\|f\|_{C^{k, \alpha}(\Omega)}$ for some constant $C$ independent of $f$ for all $|\gamma|=k$.
By  Lemma \ref{123},  for each $z_2\in D_2$, $D^\gamma T_1 f(\zeta, z_2)$ as a function of $\zeta\in D_1$ satisfies $$H_1^\alpha[D^\gamma T_1 f(\cdot, z_2)]\le\|D^{\gamma_1}_1 T_1(D^{\gamma_2}_2 f)(\cdot, z_2)\|_{C^\alpha(D_1)}
\le C\|D^{\gamma_2}_2 f\|_{C^{\gamma_1, \alpha}(\Omega)}\le C\|f\|_{C^{k, \alpha}(\Omega)}$$ for some constant $C$ independent of $f$ and $z_2$.

 On the other hand, let $z'_2(\ne z_2)\in D_2$ and consider $F_{z_2, z'_2}(\zeta): = \frac{D^{\gamma_2}_2 f(\zeta, z_2)-D^{\gamma_2}_2 f(\zeta, z'_2)}{|z_2-z'_2|^\alpha}$ on $D_1$. Since $f\in C^{k,\alpha}(\Omega)$, it follows $F_{z_2, z'_2}\in C^{\gamma_1}(D_1)$ and $\| F_{z_2, z'_2}\|_{C^{\gamma_1}(D_1)}\le \|f\|_{C^{k,\alpha}}(\Omega)$. If $\gamma_1=0$, by Lemma \ref{123}, $$\|D^{\gamma_1}_1T_1 F_{z_2, z'_2}\|_{C(D_1)} =\| T_1 F_{z_2, z'_2}\|_{C(D_1)} \le C \|F_{z_2, z'_2}\|_{ C^{0}(D_1)}  \le C\|f\|_{C^{k,\alpha}(\Omega)}, $$where $C$ is independent of $f$, $z_2$ and $z'_2$. For $\gamma_1\ge 1$,
  we have by Lemma \ref{123}, $$\|D^{\gamma_1}_1T_1 F_{z_2, z'_2}\|_{C(D_1)}\le C \|F_{z_2, z'_2}\|_{ C^{\gamma_1-1, \alpha}(D_1)} \le C\| F_{z_2, z'_2}\|_{C^{\gamma_1}(D_1)}  \le C\|f\|_{C^{k,\alpha}(\Omega)}$$ for some constant C independent of $f$, $z_2$ and $ z'_2$. In sum, for each fixed $z_1\in D_1$,  $$\frac{|D^\gamma T_1 f(z_1, z_2)-D^\gamma T_1f(z_1, z'_2)|}{|z_2-z'_2|^\alpha} = |D^{\gamma_1}_1 T_1 F_{z_2, z'_2}(z_1)|\le\|D^{\gamma_1}_1T_1 F_{z_2, z'_2}\|_{C(D_1)}\le C\|f\|_{C^{k,\alpha}(\Omega)},$$
where $C$ is independent of $f$, $z_1, z_2$ and $z'_2$. We have thus proved  $H_2^\alpha[D^\gamma T_1 f(z_1, \cdot)]\le C\|f\|_{C^{k,\alpha}(\Omega)}$ with $C$ independent of $f$ and $z_1$, and  (\ref{Tb}) as a consequence of Lemma \ref{comp}.

 \medskip

b).   As in  part a), we only prove (\ref{Sb}) for $j=1$ and $n=2$. Let $|\gamma|\le k$. Since $S_1 f$ is holomorphic with respect to $z_1$ variable, we can further assume $D^\gamma =\partial_1^{\gamma_1}D_2^{\gamma_2}$. 
 Write $\partial D_1 = \cup_{j=1}^N \Gamma_j$, where each  Jordan curve $\Gamma_j$  is connected, positively oriented  with respect to $D_1$, and  of total arclength $s_j$. 
 Let $\zeta_1(s)$ be a parameterization of $\partial D_1$ in terms of the arclength variable $s$, such that  $\zeta_1|_{ s\in [\sum_{m=1}^{j-1} s_m, \sum_{m=1}^{j} s_m)}$ is  a $C^{k+1, \alpha}$ parametrization of $\Gamma_j$.  In particular, $\bar\zeta_1' = \frac{1}{\zeta_1'}\ne 0$ on $\partial D_1$.
 For any $(z_1, z_2)\in \Omega$, it follows by integration by part,
   \begin{equation*}
  \begin{split}
   \partial_1 S_1 f(z_1, z_2)   = &\frac{1}{2\pi i}     \sum_{j=1}^N\int_{\sum_{m=1}^{j-1} s_m}^{\sum_{m=1}^{j} s_m}\partial_{z_1}(\frac{1}{\zeta_1(s)-z_1}) f(\zeta_1(s), z_2)\zeta_1'(s)ds \\
   = &-\frac{1}{2\pi i}     \sum_{j=1}^N\int_{\sum_{m=1}^{j-1} s_m}^{\sum_{m=1}^{j} s_m}\partial_{s}(\frac{1}{\zeta_1(s)-z_1}) f(\zeta_1(s), z_2)ds \\
   =&\frac{1}{2\pi i}     \sum_{j=1}^N\int_{\sum_{m=1}^{j-1} s_m}^{\sum_{m=1}^{j} s_m}\frac{\partial_{s}( f(\zeta_1(s), z_2))}{\zeta_1(s)-z_1}ds \\
   =&\frac{1}{2\pi i}     \sum_{j=1}^N\int_{\sum_{m=1}^{j-1} s_m}^{\sum_{m=1}^{j} s_m}\frac{\partial_{\zeta_1}( f(\zeta_1(s), z_2))\zeta_1'(s)+\partial_{\bar\zeta_1}( f(\zeta_1(s), z_2))\bar\zeta_1'(s)  }{\zeta_1(s)-z_1}ds \\
   =&\frac{1}{2\pi i}  \int_{\partial D_1}\frac{\partial_{\zeta_1} f(\zeta_1, z_2)+\partial_{\bar\zeta_1}( f(\zeta_1, z_2))(\bar\zeta_1'(s))^2  }{\zeta_1-z_1}d\zeta_1.
          \end{split}
  \end{equation*}

Applying the  integration by part inductively, one shall see  $D^\gamma S_1f  = S_1\tilde f$, for some function $\tilde f$ satisfying   $\|\tilde f\|_{C^{\alpha}(\Omega)}\le  \| f\|_{C^{k, \alpha}(\Omega)}$.  Therefore, we only need to prove
  $\|S_1 \tilde f\|_{C^{\alpha'}(\Omega)}\le C\|\tilde f\|_{C^{ \alpha}(\Omega)}$ for some constant $C$ independent of $\tilde f$.

Firstly,  by Lemma \ref{123}, one has $$\| S_1 \tilde f\|_{C(\Omega)} =\sup_{z_2\in D_2}\| S_1\tilde f(\cdot, z_2)\|_{C(D_1)}\le C\| \tilde f\|_{C^{ \alpha}(\Omega)}$$  for some constant $C$ independent of $\tilde f$. 

Next, we  show $H^{\alpha'}[S_1\tilde f]\le C\|\tilde f\|_{C^{ \alpha}(\Omega)}$ 
for some constant $C$ independent of $\tilde f$.
 By  Lemma \ref{123},  for each $z_2\in D_2$, $S_1 \tilde f(\zeta, z_2)$ as a function of $\zeta\in D_1$ satisfies $$H_1^{\alpha'}[S_1 \tilde f(\cdot, z_2)]\le \| S_1 \tilde f(\cdot, z_2)\|_{C^{\alpha'}(D_1)}\le C\|\tilde f\|_{C^{\alpha'}(\Omega)}\le C\|\tilde f\|_{C^{\alpha}(\Omega)}$$ for some constant $C$ independent of $\tilde f$ and $z_2$.

 We further show there exists a constant $C$ independent of $\tilde f$ and $z_1$,  such that for each $z_1\in D_1$, $H^{\alpha'}_2[S_1 \tilde f(z_1, \cdot)]\le C\|\tilde f\|_{C^{\alpha}(\Omega)}$. 
  First consider $z_1=t_1\in \partial D_1$. Without loss of generality, assume $t_1\in \Gamma_1$ with  $\zeta_1|_{s=0}=t_1$. Since $\partial D_1\in C^1$, $\partial D_1$ satisfies the so-called {\em chord-arc} condition. In other words, for any $\zeta_1(s), \zeta_1(s')\in \Gamma_j, j=1, \ldots, N$,
there exists a  constant $C\ge 1$  dependent only on $\partial D_1$ such that
\begin{equation*}
    |\zeta_1(s)-\zeta_1(s')|\le \min\{s-s', s'+s_j-s\} \le C|\zeta_1(s)-\zeta_1(s')|.
\end{equation*}
Here $s_j$ is the total arclength of $\Gamma_j$. In particular,  when $0\le s\le s_1$, \begin{equation}\label{arc}
   |d\zeta_1|\le C|ds|\ \  \text{ and}\ \  |\zeta_1(s)-t_1|\ge C\min\{s, s_1-s\}
\end{equation}
  for some constant $C$ dependent only on $D_1$.  By Sokhotski--Plemelj Formula (see \cite{Muskhelishvili} for instance), the non-tangential limit of $S_1\tilde f$ at $(t_1, z_2)\in \partial D_1\times D_2$ is
 $$\Phi_1\tilde f(t_1, z_2): = \frac{1}{2\pi i}\int_{\partial D_1}\frac{\tilde f(\zeta_1, z_2)}{\zeta_1-t_1}d\zeta_1 + \frac{1}{2}\tilde f(t_1, z_2).$$
 Here the first term is interpreted as the Principal Value.  
We shall prove that for $z_2, z'_2\in D_2$ with $h: =|z_2-z'_2|\ne 0$,
 \begin{equation*}
   |\Phi_1\tilde f(t_1, z_2) - \Phi_1\tilde f(t_1, z_2')|\le C h^{\alpha'} \|\tilde f\|_{C^{\alpha}(\Omega)}
 \end{equation*}
 for some constant $C$ independent of $ \tilde f, t_1, z_2$ and $z'_2$, essentially following the idea of  Muskhelishvili \cite{Muskhelishvili}. 

Let $h_0$ be a positive number such that $h^{\alpha-\alpha'}\ln \frac{1}{h}\le 1$ for $0<h\le h_0<\min\{1, \frac{s_1}{2}\}$. 
Then $h_0$ depends only on $\alpha$ and $\alpha'$. When $h\ge h_0$,
\begin{equation*}
   |\Phi_1\tilde f(t_1, z_2) - \Phi_1\tilde f(t_1, z_2')|\le 2 \|S_1\tilde f\|_{C^{0}(\Omega)} \le C\|\tilde f\|_{C^{\alpha}(\Omega)} \le \frac{C}{h^{\alpha'}_0}h^{\alpha'}\|\tilde f\|_{C^{\alpha}(\Omega)} \le  C h^{\alpha'} \|\tilde f\|_{C^{\alpha}(\Omega)}
 \end{equation*}
 for some constant $C$ independent of $ \tilde f, t_1, z_2$ and $z'_2$.

 When $h< h_0$, write
 \begin{equation*}
 \begin{split}
   \Phi_1\tilde f(t_1, z_2) - \Phi_1\tilde f(t_1, z_2')=& \frac{1}{2\pi i}\int_{\partial D_1}\frac{\tilde f(\zeta_1, z_2)-\tilde f(t_1, z_2) -\tilde f(\zeta_1, z'_2)+ \tilde f(t_1, z'_2)}{\zeta_1-t_1}d\zeta_1\\
   &+ \frac{\tilde f(t_1, z_2)- \tilde f(t_1, z'_2)}{2\pi i}\int_{\partial D_1}\frac{1}{\zeta_1-t_1}d\zeta_1 +\frac{\tilde f(t_1, z_2)- \tilde f(t_1, z'_2)}{2}\\
   =& \frac{1}{2\pi i}\int_{\partial D_1}\frac{\tilde f(\zeta_1, z_2)-\tilde f(t_1, z_2) -\tilde f(\zeta_1, z'_2)+ \tilde f(t_1, z'_2)}{\zeta_1-t_1}d\zeta_1
   + \\
   &+(\tilde f(t_1, z_2)- \tilde f(t_1, z'_2))\\
    =&: I+II.
 \end{split}
 \end{equation*}
Here the second equality has used the fact that $\int_{\partial D_1}\frac{1}{\zeta_1-t_1}d\zeta_1=\pi i$ when interpreted as the Principal Value, due to the positive orientation of $\partial D_1$. Obviously $$|II|\le C h^\alpha \|\tilde f\|_{C^{ \alpha}(\Omega)}$$ for some constant $C$ independent of $\tilde f, t_1, z_2$ and $z'_2$.

Let $l$ be the arc on $\partial D_1$ that are centered at $t_1$ with arclength $2h$. Consequently, $l\subset \Gamma_1$ due to the fact that $h\le \frac{s_1}{2}$. Write $I$ as follows.
 \begin{equation*}
 \begin{split}
   I = &\frac{1}{2\pi i}\int_{\Gamma_1\setminus l}\frac{\tilde f(\zeta_1, z_2)-\tilde f(t_1, z_2) -\tilde f(\zeta_1, z'_2)+\tilde f(t_1, z'_2)}{\zeta_1-t_1}d\zeta_1\\
    &+ \frac{1}{2\pi i}\int_{l}\frac{(\tilde f(\zeta_1, z_2)-\tilde f(t_1, z_2)) -(\tilde f(\zeta_1, z'_2)- \tilde f(t_1, z'_2))}{\zeta_1-t_1}d\zeta_1\\
    &+\frac{1}{2\pi i}\int_{\cup_{j=2}^N \Gamma_j}\frac{(\tilde f(\zeta_1, z_2)-\tilde f(\zeta_1, z'_2))-(\tilde f(t_1, z_2)-\tilde f(t_1, z'_2))}{\zeta_1-t_1}d\zeta_1 \\
   =&: I_1+I_2+I_3.
   \end{split}
 \end{equation*}
For $I_3$, since $\cup_{j=2}^N \Gamma_j$ does not intersect with $\Gamma_1$ and $t_1\in \Gamma_1$, $|\zeta_1-t_1|\ge C$ on $\cup_{j=2}^N \Gamma_j$ for some positive $C$ dependent only on $\partial D_1$. On the other hand, the absolute value of the numerator in  $I_3$ is less than $Ch^\alpha \|\tilde f\|_{C^{\alpha}(\Omega)}$. It immediately follows that $$|I_3|\le Ch^\alpha\|\tilde f\|_{C^{\alpha}(\Omega)}.$$  For $I_2$, the absolute value of the numerator of the integrand is less than $C|\zeta_1-t_1|^\alpha \|\tilde f\|_{C^{\alpha}(\Omega)}$. 
We infer from (\ref{arc}) that
  $$|I_2|\le C\|\tilde f\|_{C^{ \alpha}(\Omega)}\int_l\frac{1}{|\zeta_1-t_1|^{1-\alpha}}|d\zeta_1|\le C\|\tilde f\|_{C^{ \alpha}(\Omega)}\int_0^h\frac{1}{s^{1-\alpha}}ds \le Ch^\alpha\|\tilde f\|_{C^{ \alpha}(\Omega)}$$ for some constant $C$ independent of $\tilde f, t_1, z_2$ and $z'_2$.
Now we treat with the remaining term $I_1$. Rearrange  $I_1$ so it becomes
 \begin{equation*}
   |I_1| \le |\frac{1}{2\pi i}\int_{\gamma_1\setminus l}\frac{\tilde f(\zeta_1, z_2) - \tilde f(\zeta_1, z'_2)}{\zeta_1-t_1}d\zeta_1 |+| \frac{\tilde f(t_1, z_2)-\tilde f(t_1, z'_2)}{2\pi i}\int_{\gamma_1\setminus l}\frac{1}{\zeta_1-t_1}d\zeta_1|.
 \end{equation*}
 The second term of the above inequality is bounded by $Ch^\alpha\|\tilde f\|_{C^{ \alpha}(\Omega)}$ for some constant $C$ independent of $\tilde f, t_1, z_2$ and $z'_2$, as in the argument for II. The first term  when $h<h_0$ is bounded by
 \begin{equation*}
   Ch^\alpha\|\tilde f\|_{C^{\alpha}(\Omega)}\int_h^{\frac{s_1}{2}} \frac{1}{s} ds\le Ch^\alpha\ln \frac{1}{h}\|\tilde f\|_{C^{ \alpha}(\Omega)}\le Ch^{\alpha'}\|\tilde f\|_{C^{ \alpha}(\Omega)}.
 \end{equation*}

We have thus shown  there exists a constant $C$ independent of $t_1$ and $\tilde f$, such that for each $z_1=t_1\in \partial D_1$, $H_2^{\alpha'}[\Phi_1 \tilde  f(t_1, \cdot)]\le C\|\tilde f\|_{C^{\alpha}(\Omega)}$. Notice that for each fixed $\zeta\in D_2$, $S_1 \tilde f(z_1, \zeta)$ is holomorphic as a function of $z_1\in D_1$ and $C^\alpha$ continuous up to the boundary  with boundary value equal to $\Phi_1\tilde f(z_1, \zeta)$ by Plemelj--Privalov Theorem. For each fixed $z_2$ and $z_2'$  with $|z_2-z_2'|\ne 0$, applying Maximum Modulus Theorem to  the holomorphic function $\frac{S_1  \tilde f(z_1, z_2) -S_1  \tilde f(z_1, z'_2)}{|z_2-z'_2|^{\alpha'}}$ of $z_1$ in $D_1$, we immediately obtain
\begin{equation*}
  \begin{split}
   \sup_{z_1\in D_1}|\frac{S_1  \tilde f(z_1, z_2) -S_1  \tilde f(z_1, z'_2)}{|z_2-z'_2|^{\alpha'}}|    \le &\sup_{t_1\in\partial D_1}|\frac{\Phi_1  \tilde f(t_1, z_2) -\Phi_1  \tilde f(t_1, z'_2)}{|z_2-z'_2|^{\alpha'}}|\\
   =  &\sup_{t_1\in\partial D_1}H_2^{\alpha'}[\Phi_1 \tilde  f(t_1, \cdot)]\\
   \le & C\|\tilde f\|_{C^{\alpha}(\Omega)},
  \end{split}
\end{equation*}
with $C$ independent of $f, z_1, z_2$ and $z'_2$. Therefore $$H^{\alpha'}_2[S_1 \tilde f(z_1, \cdot)]\le C\|\tilde f\|_{C^{\alpha}(\Omega)}$$ with $C$ independent of $\tilde f$ and $z_1$. The proof of (\ref{Sb}) is complete.

 \end{proof}



\section{Proof of Theorem \ref{main} and Corollary \ref{mains}}

As an immediate consequence of  Theorem \ref{Holder}, we obtain the following theorem.

\begin{thm}\label{ka}
Let $D_j\subset\mathbb C$, $j= 1, \ldots, n,$ be bounded domains with  $C^{k+1,\alpha}$ boundary,   $n\ge 2, k\in \mathbb Z^+\cup \{0\}, 0<\alpha\le 1$, and let $\Omega: = D_1\times\cdots\times D_n$. Let $\mathbf f\in C_{(p, q)}^{k, \alpha}(\Omega)$, $0\le p\le n, 1\le q\le n$. Then  for any $0<\alpha'<\alpha$, $T\mathbf f$ defined in (\ref{key})
belongs to  $ C_{(p,q-1)}^{k, \alpha'}(\Omega)$ with $$\|T\mathbf f\|_{C^{k, \alpha'}(\Omega)}\le C \|\mathbf f\|_{C^{k, \alpha}(\Omega)}.$$ 
\end{thm}

\begin{proof}
 The operator $T$ defined by (\ref{key}) is well defined on $C_{(p,q)}^{k, \alpha}(\Omega)$ due to  Theorem \ref{Holder}. Choose some positive constant $\epsilon< \frac{\alpha-\alpha'}{n-1}$. Then  $\alpha'+(n-1)\epsilon< \alpha\le 1$. Applying Theorem \ref{Holder} repeatedly, it follows for each $j\le n$,
  \begin{equation*}
    \begin{split}
      \|T_jS_1\cdots S_{j-1}\pi_j \mathbf f\|_{C^{k, \alpha'}(\Omega)}\le C \|S_1\cdots S_{j-1}\pi_j \mathbf f\|_{C^{k, \alpha'}(\Omega)}    \le C \|\pi_j\mathbf f\|_{C^{k, \alpha}(\Omega)}\le C\|\mathbf f\|_{C^{k, \alpha}(\Omega)}.
    \end{split}
  \end{equation*}
 Therefore, $\|T\mathbf f\|_{C^{k, \alpha'}(\Omega)}\le C \|\mathbf f\|_{C^{k, \alpha}(\Omega)} $.

   \end{proof}

\begin{rem}
It is worth pointing out that at the top degree $q=n$, under the same assumptions as in Theorem \ref{ka},   $T\mathbf f$  
will maintain the same regularity as its data to be in  $ C_{(p,n-1)}^{k, \alpha}(\Omega)$ with $$\|T\mathbf f\|_{C^{k, \alpha}(\Omega)}\le C \|\mathbf f\|_{C^{k, \alpha}(\Omega)}.$$
This is because when $q=n$, $\mathbf f =\sum_{1\le i_1<\cdots<i_p\le n}f_{i_1\cdots i_p}dz_{i_1}\wedge\cdots\wedge dz_{i_p}\wedge d\bar z_1\wedge\cdots\wedge d\bar z_n$ for some $f_{i_1\cdots i_p}\in C^{k, \alpha}(\Omega)$. Thus 
$T\mathbf f = (-1)^p\sum_{1\le i_1<\cdots<i_p\le n}T_1f_{i_1\cdots i_p}dz_{i_1}\wedge\cdots\wedge dz_{i_p}\wedge d\bar z_2\wedge\cdots\wedge d\bar z_n\in C_{(p, n-1)}^{k, \alpha}(\Omega)$ by definition (\ref{key}). The desired estimate follows from that of $T_1$ in Theorem \ref{Holder}.
\end{rem}
\medskip

As stated in Remark \ref{mm} a), it was proved in \cite{NW} that if $\mathbf f\in C^{1}_{(p,q)}(\bar\Omega)$ is $\bar\partial$ closed, then (\ref{key})
is a solution to $\bar\partial \mathbf u = \mathbf f$ on $\Omega$. We thus have 


\begin{proof}[Proof of Corollary \ref{mains}]
Observe that $C_{(p,q)}^\infty(\bar\Omega)\subset C_{(p,q)}^{k, \alpha}(\Omega)$ for any integer $k\in \mathbb Z^+\cup\{0\}$ and $0<\alpha\le 1$.  Theorem \ref{mains} follows directly from  the proof of Theorem \ref{ka} and Remark \ref{mm} a).

\end{proof}
\medskip

\begin{rem}
When $\mathbf f\in C_{(p,1)}^{n-1, \alpha}(\Omega)$,  $T$ defined by (\ref{key}) coincides with the solution operator constructed in  \cite{ChM}\cite{FP}  by repeated application of Theorem \ref{CG}.  Therefore the same  supnorm estimate  in \cite{FP} passes onto  $T$   if the data is smooth up to the boundary. It would be interesting to know whether the supnorm estimate holds for $(p, q)$ forms smooth up to the boundary.
\end{rem}

\medskip

Assuming $\mathbf f\in C^{\alpha}(\Omega), 0<\alpha<1$, the $\bar\partial$ equation is interpreted in the sense of distributions. The following proposition shows that $T\mathbf f$ defined by (\ref{key}) solves $\bar\partial \mathbf u =\mathbf f$ in this  sense.

\begin{prop}\label{main1}
Let $D_j\subset\mathbb C,  j= 1, \ldots, n$, be  bounded domains with  $C^{1,\alpha}$ boundary, $n\ge 2$, $0<\alpha\le 1$ and $\Omega: = D_1\times\cdots\times D_n$. Assume $\mathbf f\in C_{(p,q)}^{ \alpha}(\Omega)$ is  $\bar\partial$-closed in $\Omega$ in the sense of distributions, $0\le p\le n, 1\le q\le n$. Then  $\mathbf u: = T\mathbf f$ defined in (\ref{key})  solves  $\bar\partial \mathbf u =\mathbf f$ in $ \Omega$
in the sense of distributions.
\end{prop}



\begin{proof}
  Given $\mathbf f\in \C_{(p,q)}^{\alpha}(\Omega)$  for $0<\alpha\le 1$, $T\mathbf f\in C_{(p,q-1)}^ {\alpha'}(\Omega)$ with $ 0<\alpha'<\alpha$  by Theorem \ref{ka} with $k=0$. We use the standard mollifier argument to show that $T\mathbf f$ solves $\bar\partial \mathbf u =\mathbf f$ in $\Omega$ in the sense of distributions. 

 For each $j\in \{1, \ldots, n\}$, let $ \{D^{(l)}_j\}_{l=1}^\infty$   be a family of strictly increasing open subsets  of $D_j$ such that\\
 a). for  $l\ge N_0\in \mathbb N$, $bD^{(l)}_j$ is  $C^{2, \alpha}$,  $\frac{1}{l+1}< dist(D^{(l)}_j, D_j^c)<\frac{1}{l}$;\\
 b). $H_j^{(l)}: \bar D_j\rightarrow \bar D_j^{(j)}$ is a $C^1$ diffeomorphism  with $\lim_{l\rightarrow \infty} \|H_j^{(l)}-Id\|_{C^1(D_j)}=0$.\\

 Let $\Omega^{(l)}= D^{(l)}_1\times\cdots\times D^{(l)}_n$ be the product of those planar domains.  Denote by  $T^{(l)}_j, S^{(l)}_j$ and $T^{(l)}$ the operators defined in (\ref{TS}) and (\ref{key}) accordingly, with $\Omega$ replaced by $\Omega^{(l)}$. Then $T^{(l)} \mathbf f\in C_{(p,q-1)}^{\alpha'}(\Omega^{(l)})$ for each $0<\alpha'<\alpha$. Adopting the mollifier argument to $\mathbf f\in C_{(p,q)}^{\alpha}(\Omega)$, we  obtain $\mathbf f^\epsilon\in C_{(p,q)}^{1, \alpha}( \Omega^{(l)})$ such that for each fixed $0<\alpha'<\alpha$, $\|\mathbf f^\epsilon - \mathbf f\|_{C^{\alpha'}(\Omega^{(l)})}\rightarrow 0$ (see the Appendix)  as $\epsilon\rightarrow 0$ and  $\bar\partial \mathbf f^\epsilon =0$ on $\Omega^{(l)}$.

Fix an $\alpha'(<\alpha)$. For each  $l$,  $ T^{(l)} \mathbf f^\epsilon\in C_{(p,q-1)}^{1, \alpha'}(\Omega^{(l)})$ when $\epsilon$ is small and $\bar\partial T^{(l)} \mathbf f^\epsilon =\mathbf f^\epsilon$ in $ \Omega^{(l)}$
by Theorem \ref{ka}. Furthermore, applying Theorem  \ref{Holder} at $k=0$,   we have $\|T^{(l)} \mathbf f^\epsilon - T^{(l)} \mathbf f\|_{C^{0}(\Omega^{(l)})} \le C \|\mathbf f^\epsilon - \mathbf f\|_{C^{\alpha'}(\Omega^{(l)})}\rightarrow 0$ as $\epsilon\rightarrow 0$. We thus  have  $\lim_{\epsilon\rightarrow 0}T^{(l)} \mathbf f^\epsilon$ exists in $\Omega^{(l)}$
and is equal to $T^{(l)} \mathbf f \in C_{(p,q-1)}^{\alpha'}(\Omega^{(l)})$ pointwisely. 


Given a testing $(p, q-1)$ form $\phi$ with a compact support  $K$, let $l_0\ge N_0$ be such that $K \subset \Omega^{(l_0-2)}$. 
Denote by  $(\cdot, \cdot)_{\Omega}$ (and $(\cdot, \cdot)_{\Omega^{(l_0)}}$) the inner product(s) in $L^2_{(p,q-1)}(\Omega)$ (and in $L^2_{(p,q-1)}({\Omega^{(l_0)}}$), respectively), and  $\bar\partial^*$ the formal adjoint of $\bar\partial$. 
For $l\ge l_0$, one has
\begin{equation}\label{22}
 ( T^{(l)}\mathbf f, \bar\partial^*\phi)_{\Omega^{(l_0)}} =\lim_{\epsilon \rightarrow 0}( T^{(l)}\mathbf f^\epsilon, \bar\partial^*\phi)_{\Omega^{(l_0)}}= \lim_{\epsilon \rightarrow 0}( \bar\partial T^{(l)}\mathbf f^\epsilon, \phi)_{\Omega^{(l_0)}} = \lim_{\epsilon \rightarrow 0} (\mathbf f^\epsilon, \phi)_{\Omega^{(l_0)}} = (\mathbf f, \phi)_{\Omega}.
\end{equation}

We further claim that \begin{equation}\label{11}
  ( T\mathbf f, \bar\partial^*\phi)_{\Omega}=\lim_{l\rightarrow \infty}( T^{(l)}\mathbf f, \bar\partial^*\phi)_{\Omega^{(l_0)}}.
\end{equation}
     To prove this,  for simplicity of notations yet without loss of generality, assume $\pi_j \mathbf f $ contains only one component function $f_j$, so is for $\phi$. We will also drop various integral measure, which is clear from context. For each $j\ge 1$,
\begin{equation*}
  \begin{split}
    &-(-2i)^n(2\pi i)^j( T_j^{(l)}S_1^{(l)}\cdots S_{j-1}^{(l)}\pi_j \mathbf f, \bar\partial^*\phi)_{\Omega^{(l_0)}} \\
     =&\int_{z\in K}\int_{\zeta_j\in D_j^{(l)}}\int_{\zeta_1\in\partial D_1^{(l)}}\cdots \int_{\zeta_{j-1}\in \partial D_{j-1}^{(l)}}\frac{f_j(\zeta_1, \cdots, \zeta_j, z_{j+1}, \cdots, z_n)\overline{\bar\partial^*\phi(z)}}{(\zeta_1-z_1)\cdots(\zeta_j-z_j)}\\
    = &\int_{(z, \zeta_j)\in K\times D_j}\left(\int_{\zeta_1\in\partial D_1^{(l)}}\cdots \int_{\zeta_{j-1}\in \partial D_{j-1}^{(l)}}\frac{f_j(\zeta_1, \cdots, \zeta_j, z_{j+1}, \cdots, z_n)\chi_{D_j^{(l)}}(\zeta_j)\overline{\bar\partial^*\phi(z)}}{(\zeta_1-z_1)\cdots(\zeta_j-z_j)}\right).
  \end{split}
\end{equation*}
Here $\chi_{D_j^{(l)}}$ is the step function on $\mathbb C$ such that  $\chi_{D_j^{(l)}}=1$ in $D_j^{(l)}$ and 0 otherwise.  

Firstly,  as a function of $(z, \xi_j)\in K\times D_j$,  $$\int_{\zeta_1\in\partial D_1^{(l)}}\cdots \int_{\zeta_{j-1}\in\partial D_{j-1}^{(l)}}\frac{f_j(\zeta_1, \cdots, \zeta_j, z_{j+1}, \cdots, z_n)\chi_{D_j^{(l)}}(\zeta_j)\overline{\bar\partial^*\phi(z)}}{(\zeta_1-z_1)\cdots(\zeta_j-z_j)}\in L^1(K\times D_j).$$ To see this, notice that if  $z\in K(\subset \Omega^{(l_0-2)})$ and $\zeta_k\in \partial D_k^{(l)}, l\ge l_0, k=1, \ldots, j-1$, then $$|\zeta_k-z_k|\ge dist((\Omega^{(l)})^c, \Omega^{(l_0-2)})\ge dist((\Omega^{(l_0)})^c, \Omega^{(l_0-2)})> \frac{1}{l_0^2}: =\delta_0.$$
Hence  for each $(z, \zeta_j)\in K\times D_j\setminus \{z_j=\zeta_j\}$,
 $$\left|\int_{\zeta_1\in \partial D_1^{(l)}}\cdots \int_{\zeta_{j-1}\in \partial D_{j-1}^{(l)}}\frac{f_j(\zeta_1, \cdots, \zeta_j, z_{j+1}, \cdots, z_n)\chi_{D_j^{(l)}}(\zeta_j)\overline{\bar\partial^*\phi(z)}}{(\zeta_1-z_1)\cdots(\zeta_j-z_j)}\right|\le \frac{C}{\delta_0^{j-1}|\zeta_j-z_j|}$$  for some constant $C>0$, which is integrable in $K\times D_j$. 
 
 On the other hand,  by continuity of $f_j$ and the construction of $\Omega^{(l)}$,
\begin{equation*}
  \begin{split}
    &\lim_{l\rightarrow \infty} \int_{\zeta_1\in\partial D_1^{(l)}}\cdots \int_{\zeta_{j-1}\in\partial D_{j-1}^{(l)}}\frac{f_j(\zeta_1, \cdots, \zeta_j, z_{j+1}, \cdots, z_n)\chi_{D_j^{(l)}}(\zeta_j)\overline{\bar\partial^*\phi(z)}}{(\zeta_1-z_1)\cdots(\zeta_j-z_j)}\\
    = &\int_{\zeta_1\in\partial D_1}\cdots \int_{\zeta_{j-1}\in\partial D_{j-1}}\frac{f_j(\zeta_1, \cdots, \zeta_j, z_{j+1}, \cdots, z_n)\overline{\bar\partial^*\phi(z)}}{(\zeta_1-z_1)\cdots(\zeta_j-z_j)}
  \end{split}
\end{equation*}
pointwisely in $K\times D_j$. Applying Dominated Convergence Theorem, we obtain
 \begin{equation*}
 \begin{split}
&\lim_{l\rightarrow \infty}-(-2i)^n(2\pi i)^j( T_j^{(l)}S_1^{(l)}\cdots S_{j-1}^{(l)}\pi_j \mathbf f, \bar\partial^*\phi)_{\Omega^{(l_0)}} \\
=& \int_{(z, \zeta_j)\in K\times D_j}\int_{\zeta_1\in\partial D_1}\cdots \int_{\zeta_{j-1}\in\partial D_{j-1}}\frac{f_j(\zeta_1, \cdots, \zeta_j, z_{j+1}, \cdots, z_n)\overline{\bar\partial^*\phi(z)}}{(\zeta_1-z_1)\cdots(\zeta_j-z_j)}  \\
=&-(-2i)^n(2\pi i)^j (T_jS_1\cdots S_{j-1}\pi_j \mathbf f, \bar\partial^*\phi)_{\Omega}.
\end{split}
\end{equation*}
 (\ref{11}) is thus proved for $T$ in view of  its definition (\ref{key}).

Finally, combining  (\ref{22}) with (\ref{11}),  we deduce that
\begin{equation*}
  (\bar\partial T\mathbf f, \phi)_\Omega  = ( T\mathbf f, \bar\partial^*\phi)_{\Omega}=\lim_{l\rightarrow \infty}( T^{(l)}\mathbf f, \bar\partial^*\phi)_{\Omega^{(l_0)}} =(\mathbf f, \phi)_\Omega.
\end{equation*}
The proof of Proposition \ref{main1} is complete.

\end{proof}
\medskip

\begin{proof}[Proof of Theorem \ref{main}] Theorem \ref{main} follows directly from Theorem \ref{ka},  Remark \ref{mm} a)  and Proposition \ref{main1}.

\end{proof}
\medskip

Finally, making use of the idea of Kerzman \cite{Kerzman}, we argue by the following  examples the regularity of the $\bar\partial$ solution  can not be improved in H\"older spaces over product domains.
\medskip


\begin{example}\label{ex2}
a). For each $k\in \mathbb Z^+\cup \{0\}$ and $ 0<\alpha<1$, consider $\bar\partial u =\mathbf f:= \bar\partial ((z_1-1)^{k+\alpha}\bar z_2)$ on $\mathbb D^2$, $\frac{1}{2}\pi <\arg (z_1-1)<\frac{3}{2}\pi$. Then $\mathbf f= (z_1-1)^{k+\alpha}d\bar z_2\in C^{k,\alpha}(\mathbb D^2)$ is a $\bar\partial$-closed $(0, 1)$ form. However, there does not exist a solution $u\in C^{k, \alpha'}(\mathbb D^2)$ to $\bar\partial u =\mathbf f$ on $\mathbb D^2$ for any $\alpha'>\alpha$.
\medskip

b). For each $k\in \mathbb Z^+\cup \{0\}$, consider $\bar\partial u =\mathbf f:= \bar\partial (\frac{(z_1-1)^{k+1}}{\log(z_1-1)}\bar z_2)$ on $\mathbb D^2$, $\frac{1}{2}\pi <\arg (z_1-1)<\frac{3}{2}\pi$. Then $\mathbf f = \frac{(z_1-1)^{k+1}}{\log(z_1-1)}d\bar z_2\in C^{k,1}(\mathbb D^2)$ is a $\bar\partial$-closed $(0,1)$ form. However, there does not exist a solution $u\in C^{k+1, \alpha}(\mathbb D^2)$ to $\bar\partial u =\mathbf f$ on $\mathbb D^2$ for any $\alpha>0$. 
 \end{example}
\medskip

\begin{proof} a).   $\mathbf f$ is well defined in $\mathbb D^2$ and $\mathbf f = (z_1-1)^{k+\alpha} d\bar z_2\in C^{k, \alpha}(\mathbb D^2)$. Assume by contradiction that there exists a solution $u\in C^{k, \alpha'}(\mathbb D^2)$ to $\bar\partial u =\mathbf f$ in $\mathbb D^2$ for some $\alpha'$ with $\alpha <\alpha'<1$. Then $u = h +(z_1-1)^{k+\alpha}\bar z_2$ for some holomorphic function $h$ in $\mathbb D^2$.

  Consider $w(\xi): =\int_{|z_2|=\frac{1}{2}}u(\xi, z_2)dz_2$ for $\xi\in \mathbb D: =\{z\in \mathbb C: |z|<1\}$. Since $u\in C^{k, \alpha'}(\mathbb D^2)$, we have $w\in C^{k, \alpha'}(\mathbb D)$ as well. On the other hand, by Cauchy's Theorem,
  \begin{equation*}
     w(\xi) =\int_{|z_2|=\frac{1}{2}} (\xi-1)^{k+\alpha}\bar z_2dz_2 = (\xi-1)^{k+\alpha}\int_{|z_2|=\frac{1}{2}} \frac{1}{4 z_2}dz_2 = \frac{\pi i}{2}(\xi-1)^{k+\alpha}.
  \end{equation*}
  This is a contradiction since $(\xi-1)^{k+\alpha}\notin C^{k, \alpha'}(\mathbb D)$ for any $\alpha'>\alpha$.
  \medskip
  
  b).   Argue in a similar way as in a) by noticing that  $\mathbf f = \frac{(z_1-1)^{k+1}}{\log (z_1-1) }d\bar z_2\in C^{k, 1}(\mathbb D^2)$. If $u\in C^{k+1, \alpha}(\mathbb D^2)$ solves $\bar\partial u =\mathbf f$ in $\mathbb D^2$ for some $\alpha>0$, then $u = h +\frac{(z_1-1)^{k+1}}{\log (z_1-1)}\bar z_2$ for some holomorphic function $h$ in $\mathbb D^2$ and  
  $w(\xi): =\int_{|z_2|=\frac{1}{2}}u(\xi, z_2)dz_2\in C^{k+1, \alpha}(\mathbb D^2)$. However by Cauchy's Theorem,
  \begin{equation*}
        w(\xi) =\int_{|z_2|=\frac{1}{2}} \frac{(\xi-1)^{k+1}}{\log (\xi-1)}\bar z_2dz_2 =\frac{\pi i}{2}\frac{(\xi-1)^{k+1}}{\log (\xi-1)}\notin C^{k+1, \alpha}(\mathbb D) 
    \end{equation*}
 for any $\alpha>0$.
 
  \end{proof}

\appendix
\section{Appendix}
Let $\Omega\subset\mathbb R^n$ be a bounded domain,  $\Omega_j: = \{x\in \Omega: dist(x, \partial \Omega)>\frac{1}{j}\}$ when $j$ is large, and $\rho$ be a smooth function  in $\mathbb R^n$ by $$
 \rho(x):  =
  \left\{
      \begin{array}{cc}
     C\exp(\frac{1}{|x|^2-1}), & |x|<1;\\
      0, & |x|\ge 1,
         \end{array}
\right.$$
where $C$ is selected such that $\int_{\mathbb R^n}\rho(y) dy =1$. $\rho$ is called the standard mollifier.  Let   $f\in L^1_{loc}(\Omega)$ and define for $x\in \Omega_j$,
\begin{equation}\label{ap}
f_j(x): =\int_{|y|\le 1}\rho(y)f(x-\frac{y}{j})dy. \end{equation} Then $f_j\in C^\infty(\Omega_j)$. The mollifier argument is a standard method dealing with weak derivatives in Sobolev spaces (See, for instance, \cite{Evans} p. 717). The following theorem ought to be well-known for H\"older spaces, however we could not locate a reference.  For convenience of the reader, we include the proof below.

\begin{thm}
Let $\tilde \Omega\subset\subset \Omega$ and  $0<\alpha'<\alpha$. If $f\in C^{\alpha}(\Omega)$, then $f_j\rightarrow f$ in $C^{\alpha'}(\tilde \Omega)$. I.e., $\|f_j-f\|_{C^{\alpha'}(\tilde \Omega)}\rightarrow 0$ as $j\rightarrow \infty$.
\end{thm}
\begin{proof}
Let $j_0$ be such that $\tilde \Omega\subset \Omega_{j_0}$ and assume $ j\ge j_0$.  $\|f_j-f\|_{C(\tilde \Omega)}\rightarrow 0$  due to the uniform continuity of $f$ on $\Omega$ (\cite{Evans} p.718). Write $\phi_j(x): = f_j(x)-f(x) = \int_{|y|\le 1}\rho(y)(f(x-\frac{y}{j})-f(x))dy$. We next show for any $\epsilon>0$, there exists $N\in \mathbb N$ such that when $j\ge N$,
\begin{equation*}
    \frac{|\phi_j(x)-\phi_j(x')|}{|x-x'|^{\alpha'}}\le \epsilon,
\end{equation*}
for all $x, x'\in \tilde \Omega$. Indeed, choose $\delta_0>0$ satisfing $\|f\|_{C^\alpha(\Omega)}\delta_0^{\alpha-\alpha'}\le \frac{\epsilon}{2}$.

When $|x-x'|\le \delta_0$,
\begin{equation*}
\begin{split}
    \frac{|\phi_j(x)-\phi_j(x')|}{|x-x'|^{\alpha'}}&\le \int_{|y|\le 1}\rho(y)\frac{|f(x-\frac{y}{j})-f(x'-\frac{y}{j})|}{|x-x'|^{\alpha'}}dy +  \int_{|y|\le 1}\rho(y)\frac{|f(x)-f(x')|}{|x-x'|^{\alpha'}}dy\\
    &\le 2\|f\|_{C^\alpha(\Omega)}|x-x'|^{\alpha-\alpha'}\le \epsilon.
    \end{split}
    \end{equation*}

    When $|x-x'|> \delta_0$, choose $N\in\mathbb N$ such that $\| f\|_{C^\alpha(\Omega)}\delta_0^{-{\alpha'}}N^{-\alpha}\le \frac{\epsilon}{2}$. Then for any $j\ge N$, $|x-x'|> \delta_0$, we have
    \begin{equation*}
\begin{split}
    \frac{|\phi_j(x)-\phi_j(x')|}{|x-x'|^{\alpha'}}&\le \int_{|y|\le 1}\rho(y)\frac{|f(x-\frac{y}{j})-f(x)|}{|x-x'|^{\alpha'}}dy +  \int_{|y|\le 1}\rho(y)\frac{|f(x'-\frac{y}{j})-f(x')|}{|x-x'|^{\alpha'}}dy\\
    &\le 2\|f\|_{C^\alpha(\Omega)} |x-x'|^{-\alpha'}j^{-\alpha}\le \epsilon.
    \end{split}
    \end{equation*}
\end{proof}
 Given $f\in C^\alpha(\Omega)$, although only the  $C^{\alpha'}$ convergence of the family $\{f_j\}$ defined by (\ref{ap}) for some $\alpha'>0$ is needed  in Proposition \ref{main1},  we note  that the $C^{\alpha}$ convergence of $\{f_j\}$ can not be achieved in general. The following simple counter-example was provided by Liding Yao.

\begin{example}
Let $\Omega=(-1, 1)\in \mathbb R$ and $$
 f(x) =
  \left\{
      \begin{array}{cc}
      0, &  x\le 0;\\
     x^\alpha , & x>0.
         \end{array}
\right.$$
Then $f\in C^\alpha(\Omega)$. However, for any $\tilde \Omega\subset\subset \Omega$ containing the origin, $\|f_j-f\|_{C^\alpha(\tilde \Omega)}\ge \int_{0}^1\rho(y)y^\alpha dy>0$ for sufficiently large $j$.
\end{example}

\begin{proof}
Let $j_0$ be such that $\tilde \Omega\subset \Omega_{j_0}$ and assume $ j\ge j_0$. Write $\phi_j(x): = f_j(x)-f(x) = \int_{-1}^1\rho(y)(f(x-\frac{y}{j})-f(x))dy$. For each fixed $j$, it can be verified that  $$\phi_j(-\frac{1}{j}) =\int_{-1}^1\rho(y)f(-\frac{1+y}{j})dy = 0$$ and $$\phi_j(0) =\int_{-1}^1\rho(y)f(-\frac{y}{j})dy =  (\frac{1}{j})^\alpha\int_{0}^1\rho(y)y^\alpha dy.$$
However  for all $j$, $$\|\phi_j\|_{C^\alpha(\tilde \Omega)}\ge  \frac{\phi_j(0) -\phi_j(-\frac{1}{j})}{(\frac{1}{j})^\alpha} = \int_{0}^1\rho(y)y^\alpha dy>0.$$
\end{proof}

 \noindent Yifei Pan, pan@pfw.edu, Department of Mathematical Sciences, Purdue University Fort Wayne,
Fort Wayne, IN 46805-1499, USA \medskip

\noindent Yuan Zhang, zhangyu@pfw.edu, Department of Mathematical Sciences, Purdue University Fort Wayne,
Fort Wayne, IN 46805-1499, USA

\end{document}